\documentclass[a4paper]{amsart}
\usepackage{fullpage}
\usepackage[dvipsnames,svgnames,x11names,hyperref]{xcolor}
\usepackage[pagebackref,colorlinks,citecolor=Blue,linkcolor=Blue,urlcolor=Blue,filecolor=Blue]{hyperref}
\usepackage[all]{xy}
\usepackage{url, amsmath, amsthm, amssymb, accents, bm, graphicx}
\usepackage[capitalise]{cleveref}
\usepackage{setspace}
\usepackage{stmaryrd}
\usepackage{tikz}
\usetikzlibrary{matrix,arrows,calc}
\usepackage{caption}
\usepackage{subcaption}
\usepackage{mathtools}  
\usepackage{tabulary}

\title{Homological stability for subgroups of surface braid groups}
\author{TriThang Tran}
\thanks{The author was partially supported by a David Hay postgraduate award}
\email{trithang@uoregon.edu}
\date{\today}

\makeatletter
\newtheorem*{rep@theorem}{\rep@title}
\newcommand{\newreptheorem}[2]{%
\newenvironment{rep#1}[1]{%
 \def\rep@title{#2 \ref{##1}}%
 \begin{rep@theorem}}%
 {\end{rep@theorem}}}
\makeatother

\newreptheorem{theorem}{Theorem}
\newreptheorem{lemma}{Lemma}

\newtheorem{theorem}{Theorem}[section]
\newtheorem{lemma}[theorem]{Lemma}

\newtheorem{proposition}[theorem]{Proposition}
\newtheorem{definition}[theorem]{Definition}

\theoremstyle{remark}
\newtheorem{remark}[theorem]{Remark}

\newcommand{\del}{\partial}

\newcommand{\norm}[1]{\lVert#1\rVert}
\newcommand{\st}{\; \vert \;}

\newcommand{\conf}{\mathrm{Conf}}
\newcommand{\pconf}{\mathrm{PConf}}

\newcommand{\stab}{\mathrm{stab}}
\newcommand{\lf}{\left\lfloor}
\newcommand{\rf}{\right\rfloor}
\newcommand{\surf}{S}
\newcommand{\braid}{\mathrm{Br}}

\newcommand{\sym}{\Sigma}
\newcommand{\symp}{\mathrm{Sym}}

\newcommand{\forget}{\mathrm{forget}}

\begin{document}

\maketitle

\begin{abstract}
In this paper we prove homological stability for certain subgroups of surface braid groups. Alternatively, this is equivalent to proving homological stability for configurations of subsets of exactly $\xi$ points in a surface as we increase the number of subsets. For open surfaces, we prove the result integrally using a variation of the arc complex which we dub the ``fern complex". We use a technique of Randal-Williams to extend the result rationally for closed surfaces.
\end{abstract}

\section{Introduction} \label{sec: introduction}
Let $\surf$ be a connected surface and $\xi \in \mathbb{Z}_{\geq 1}$. The \emph{configuration space of $\xi$ points in $\surf$} is the space of subsets of $\surf$ of size exactly $\xi$. A \emph{surface braid group} is a fundamental group, $\braid_\xi(\surf):= \pi_1(\conf_{\xi}(\surf))$. When $\surf = D^2$ is a disk, then this is the usual Artin braid group with $\xi$ strands. In this paper, we are interested in studying particular subgroups of surface braid groups which we define as follows.

The \emph{$n$th ordered $\xi$-configuration space of $\surf$} is
\[ \pconf_n^\xi(\surf) := \{ ( \bm{p}_1, \ldots, \bm{p}_n ) \subset \conf_\xi(\surf)^{\times n} \st \bm{p}_i \cap \bm{p}_j = \emptyset \mbox{ for } i \neq j \}. \]
The symmetric group $\Sigma_n$ acts on $\pconf_n^\xi(\surf)$ by: if $\sigma \in \sym_n$ then $\sigma.(\bm{p}_1, \ldots, \bm{p}_n) = (\bm{p}_{\sigma(1)}, \ldots, \bm{p}_{\sigma(n)})$. 
\begin{definition} The \emph{$\xi$-configuration space} of $\surf$ is the quotient
\[ \conf_n^\xi(\surf) := \frac{\pconf_n^\xi(\surf)}{\Sigma_n}. \]
\end{definition}
One way to view this space is that it is the usual configuration space, $\conf_{n\xi}(\surf)$, of $n\xi$ points in $\surf$, where the points have been partitioned into $n$ subsets of size $\xi$. There is a covering map
	\[ \conf_n^\xi(\surf) \to \conf_{n\xi}(\surf) \]
that forgets the partition of points into subsets. This is a finite sheeted covering map whose fibres correspond to the number of ways to group $n \xi$ points into subsets of size $\xi$. In our pictures, we will represent points being in different subsets by using different shapes. It is also natural to think of points being ``coloured".

\begin{definition}
The \emph{$\xi$-surface braid group} of $\surf$ is the fundamental group
	\[ \braid^\xi_n(\surf) := \pi_1(\conf_n^\xi(\surf)). \]
\end{definition}
In \cite{fn62}, Fadell and Neuwirth show that $\conf_{n\xi}(\surf)$ is $K(\pi,1)$. The long exact sequence in homotopy groups then shows that $\conf_n^\xi(\surf)$ is also $K(\pi,1)$. 
There is therefore an isomorphism 
	\[ H_*(\conf_n^\xi(\surf) ; \mathbb{Z}) \cong H_*(\braid_n^\xi(\surf) ; \mathbb{Z}) \]
between singular homology and group homology. The upshot of this isomorphism is that studying the homology of $\conf_n^\xi(\surf)$ is the same as studying the homology of $\braid_n^\xi(\surf)$. We will freely switch between the two. The goal of this paper is to prove the following two theorems.

\begin{reptheorem}{thm: surface braid stability} Let $\surf$ be the interior of a surface with boundary. There is a map $\stab : \braid^\xi_n(\surf) \rightarrow \braid^\xi_{n+1}(\surf)$ such that the induced map
	\[ \stab_*: H_k(\braid^\xi_n(\surf) ; \mathbb{Z}) \rightarrow H_k(\braid^\xi_{n+1}(\surf) ; \mathbb{Z}) \]
is an isomorphism for $2k \leq n$.
\end{reptheorem}

\begin{reptheorem}{thm: surface braid closed} Let $\surf$ be any surface (open or closed).
	\[ H_*(\braid^\xi_n(\surf) ; \mathbb{Q}) \cong H_*(\braid^\xi_{n+1}(\surf) ; \mathbb{Q}) \]
for $2k \leq n$. The isomorphism can be realised as a transfer map \[t_{n+1} : H_k(\braid_{n+1}^\xi (\surf); \mathbb{Q}) \to H_k(\braid_n^\xi(\surf);\mathbb{Q})\] which we define in \cref{sec: braid stability closed}.
\end{reptheorem}

The main technical result that we will use is a high connectivity result for certain simplicial complexes which we call \emph{fern complexes} and are defined in \cref{sec: fern complex}. The groups $\braid_n^\xi(\surf)$ will act on these fern complexes and from there, the literature has a well oiled machine for proving homological stability. In \cref{sec: open}, we will make use of an axiomatisation of this machine by Hatcher and Wahl, which is Theorem 5.1 of \cite{hw10} to prove \cref{thm: surface braid stability}.

\subsection{A brief history} Homological stability for braid groups was first proven by Arnol'd in \cite{arnold69}. Homological stability was then extended to configuration spaces of open manifolds by the work of Segal and McDuff in the 70's \cite{segal73, mcduff75, segal79}. Randal-Williams and Church have recently shown that configuration spaces of closed manifolds satisfy homological stability with rational coefficients \cite{orw13, church12}. This paper studies the homology of natural covering space of configuration spaces. Thus the main theorems can be regarded as a homological stability theorem for braid groups with certain twisted coefficients. 

Configuration spaces for manifolds other than points have also been studied in the literature. For example, the configuration space of circles in $\mathbb{R}^3$ has been well studied (see for example \cite{bh13}). Its fundamental group is sometimes called the symmetric automorphism group, $\sym \mathrm{Aut} F_n$, or just the ring group $R_n$. Homological stability for $R_n$ was proven by Hatcher and Wahl in \cite{hw10}. Homological stability for $\xi$-configuration spaces gives an example of homological stability for manifolds where the manifold being stabilised by is disconnected. In his doctoral thesis, Palmer proves homological stability for configurations of (possibly disconnected) manifolds in a larger ambient manifold under certain dimension conditions \cite{palmer13b}. However, the dimesnion conditions do not apply to the case of $0$-manifolds in a surface so that this paper serves also to fill this curious gap in the literature.

\subsection{Outline}
In \cref{sec: fern complex} we define a simplicial complex on which $\braid_n^\xi(\surf)$ will act and show that this simplicial complex is highly connected. In \cref{sec: stab} we give a definition of the stabilisation map for $\xi$-configuration spaces, which can be appropriately identified with the stabilisation map for $\xi$-braids. In \cref{sec: open} we prove \cref{thm: surface braid stability}. In \cref{sec: braid stability closed} we define the transfer map for $\xi$-configuration spaces and use it to prove \cref{thm: surface braid closed}.

\subsection{Acknowledgements}
I would like to thank Nathalie Wahl and Craig Westerland for many useful discussions relating to this work. I would also like to thank Jeffrey Bailes for reading earlier drafts of this paper. This work formed part of my work towards a Ph.D at the University of Melbourne. It was a wonderful place to do research.

\section{Fern complexes} \label{sec: fern complex}
In this section we will define the simplicial complexes that our $\xi$-braids will act on and show that they are highly connected. We will make use of some of the theory of simplicial complexes which is well summarised in the appendix of \cite{wahl12}.

Let $\xi \in \mathbb{Z}_{\geq 1}$ be fixed. Let $\surf$ be a surface with boundary with $n\xi$ marked points in its interior. Further let $\Delta_1, \ldots, \Delta_n$ be a partition of those marked points into disjoint sets of size $\xi$. Lastly, fix a base point $*$ in the boundary $\del \surf$. By an \emph{arc}, we will mean an embedded path in $\surf$ with one endpoint meeting the boundary of $D$ at $*$ transversally and the other at one of the marked points in $\surf$. Simplicial complexes where simplices are defined by non-intersecting arcs, called \emph{arc complexes} and have been studied extensively in the literature (see for example \cite{harer85, hatcher91}). We will study an analogous version of this which we call the ``fern complex".


\begin{definition} A \emph{fern} is an unordered $\xi$-tuple of arcs that do not intersect except at $*$ and such that their other endpoints are all in the same $\Delta_i$.
\end{definition}

We will be considering ferns up to isotopy relative to the marked points and $*$. Two or more isotopy classes of ferns are \emph{disjoint} if there exists representative ferns that are disjoint.

\begin{definition} \label{def: fern} The \emph{fern complex}, $A_n^\xi= A_n^\xi(\surf, \Delta_1, \ldots \Delta_n)$, is the simplicial complex such that:
	\begin{itemize}
	\item vertices are isotopy classes of ferns; and
	\item $p$-simplices are collections of $(p+1)$ vertices that are disjoint except at $*$.
	\end{itemize}
\end{definition}
See \cref{pic: fern simplex} for examples of simplices in the fern complex of a disk.
\begin{center}
\begin{figure}
		\includegraphics[scale=0.26]{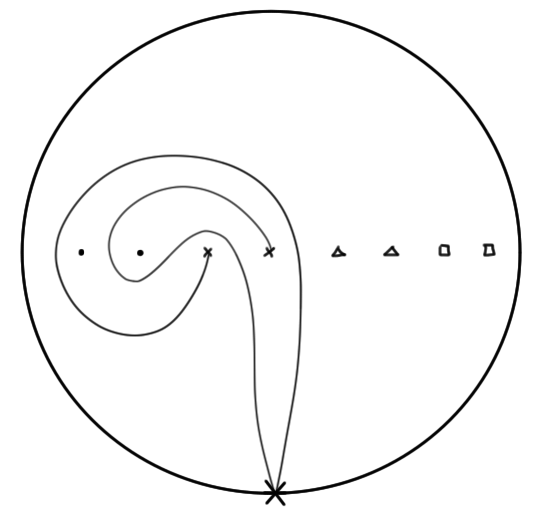}
		\includegraphics[scale=0.26]{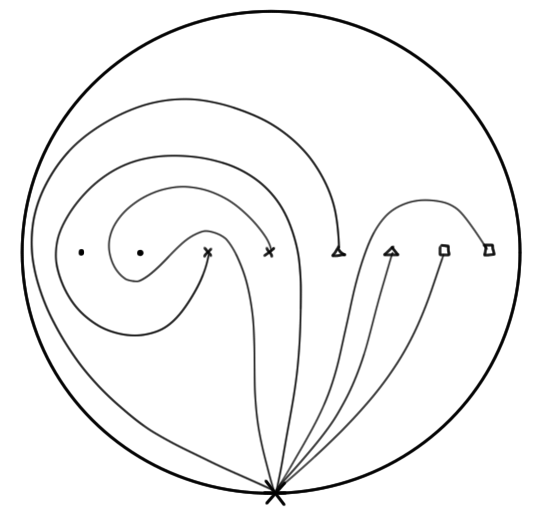}
\caption{A picture of representatives of a vertex (left) and a 2-simplex (right) in $A_4^2(D^2)$. Note that a single fern must have all endpoints of arcs that are not at $*$ in the same subset.} \label{pic: fern simplex}
\end{figure}
\end{center}
\vspace{-10pt}

Recall that a space $X$ is $n$-connected if its homotopy groups, $\pi_k(X)$, are trivial for $k \leq n$. The goal of this section is to prove the following connectivity theorem for fern complexes.
\begin{theorem} \label{thm: fern connectivity} $A_n^\xi$ is $n-2$ connected.
\end{theorem}
Our method of proof will be similar to the connectivity proofs found in Section 4 of \cite{wahl12}. In order to prove \cref{thm: fern connectivity}, we will need to study the following related simplicial complex, where we allow ferns to also agree at marked points.

\begin{definition} \label{def: fan fern} $FA_n^\xi = FA_n^\xi(\surf, \Delta_1, \ldots \Delta_n)$ is the simplicial complex such that:
	\begin{itemize}
	\item vertices are isotopy classes of ferns; and
	\item $p$-simplices are collections of $(p+1)$ vertices that are disjoint except possibly at their endpoints.
	\end{itemize}
\end{definition}
The following is an analogue to a special case of the main theorem of Hatcher in \cite{hatcher91}, which proves contractibility for certain arc complexes of surfaces.

\begin{lemma} \label{lem: fern contractible} $FA_n^\xi$ is contractible.
\end{lemma}
\begin{proof}
Fix a vertex $v$ of $FA_n^\xi$, and a representative fern, which we also call $v$, with arcs going from $*$ to $\Delta_1$ say. Moreover fix an ordering of the arcs of $v = (v_1, \ldots, v_\xi)$. We will show that $FA_n^\xi$ deformation retracts onto $\mathrm{Star}(v)$. Order the interior points of $v$ so that 
\begin{enumerate}
	\item $x \prec y$ if $x \in v_i, y \in v_j,$ and $i < j$; 
	\item If $i = j$ then $x\prec y$ if $x$ is closer to $*$ along $v_i$ than $y$.
\end{enumerate}

Let $\sigma = \langle a_0, \ldots, a_k \rangle$ be a simplex of $FA_n^\xi$. In particular the $a_i$ are $\xi$-tuples of arcs. Choose representatives ferns of $\sigma$ so that $a_0 \cup \ldots \cup a_k$ intersect $v$ minimally.

Consider the ferns of $\sigma$ that intersect $v$. Suppose that there are $k$ points of intersection and denote by $g_1, \ldots , g_k$ the germs of arcs that are constituents of ferns corresponding to those intersection points. By germ, we simply mean a small arc segment that intersects $v$. Moreover assume that $g_i \prec g_j$ for $i<j$, where $g_i \prec g_j$ if their corresponding intersection points with $v$ satisfy the inequality $\prec$. The $g_i$ are germs of the arcs of ferns $a_{j_i}$, where it is possible that $j_i = j_{i'}$ for $i \neq i'$ if the arc intersects $v$ more than once.

We will now describe a sequence of $k$, $(p+1)$-simplices $r_1(\sigma), \ldots ,r_k(\sigma)$ associated to $\sigma$. Our retraction will then simply be a map that carries $\sigma$ through these simplices.

If $\alpha_i$ is an arc intersecting the arc $v_i$ of a fern $v$ at a point $x$, and $x$ is the first intersection point of $\alpha_i$ according to the ordering $\prec$ of $v$, we can define $L(\alpha_i)$ and $R(\alpha_i)$ to be the new path obtained by cutting $\alpha_i$ at $x$ and joining the new endpoints to $*$ by travelling along the left and right hand side of $v_i$. Note that one of the paths of $L(\alpha_i)$ or $R(\alpha_i)$ will not be an arc since it will have one path with both endpoints at $*$. Call $C(\alpha_i)$ the one that is an arc. 

Now if $a = (\alpha_1, \ldots, \alpha_\xi )$ is a fern, we define $C(a)$ to be the new fern obtained by doing $C$ to the arc intersecting $v$ ``first" according to the ordering $\prec$ on $v$. For example, if $\alpha_1$ had an intersection with $v$ that occurred first according to $\prec$, then $C(a) = (C(\alpha_1), \alpha_2, \ldots , \alpha_\xi)$. Note that $C(\alpha_1)$ will still be a fern whose endpoints still go to the same $\Delta_i$ as $\alpha_1$ since endpoints do not change when doing $C$.

We will now use the operator $C$ on ferns to define our sequence of simplices. Let $r_i(\sigma)$ be the $(p+1)$-simplex given by
\[ r_i(\sigma) = \langle b_0, \ldots , b_{p+1} \rangle, \]
where $b_l = C^{\epsilon_i(l)}(a_l)$ for $l \leq p$, and $b_{p+1} = L^{\epsilon_i(j_i) + 1}(a_{j_i}) = L^{\epsilon_{i+1}(j_i)}(a_{j_i})$, and $\epsilon_i(l)$ is the number $j < i$ such that $g_j$ is a germ of $a_l$. \cref{pic: fern contract} gives an example for $\xi=2$ of this retraction process.
\begin{center}
\begin{figure}
		\includegraphics[scale=0.42]{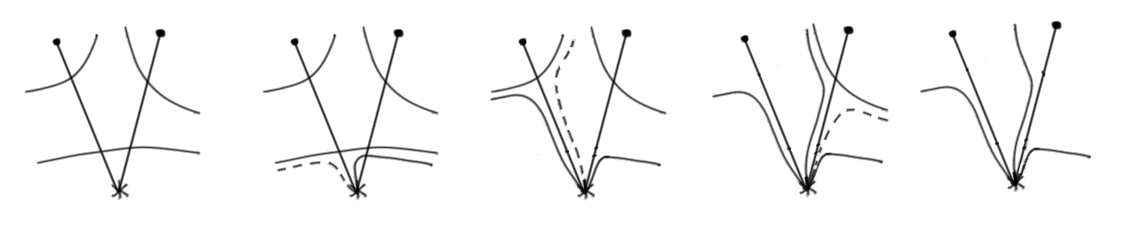}
\caption[A retraction of a simplex in $FA_{n,\xi}$]{An example (for $\xi = 2$) of germs of ferns of a that intersect $v$, on the left followed by the simplices $r_1, r_2, r_3$ of the retraction in the middle and finally the image of germs on the right. The dotted lines represent the discarded part of arcs during the cutting process.} \label{pic: fern contract}
\end{figure}
\end{center}
\vspace{-30pt}
Using barycentric coordinates, a point on $\sigma$ can be identified with the $p$-tuple $(t_0, \ldots, t_p)$, such that $\sum t_i = 1$. We interpret these coordinates as the fern $a_i$ having the weight $t_i$. Assign to the $i$th germ $g_i$ the weight $w_i = t_{j_i}/2$. For $\sum_{j=1}^{i-1} w_j \leq s \leq \sum_{j=1}^{i}w_j$, define $f: I \times FA_n \rightarrow FA_n$ by 
\[ f(s, [ \sigma, (t_0, \ldots, t_p)]) = [r_i(\sigma), (v_0, \ldots, v_{p+1})],\]
where the weight $v_i = t_i$ except for the pair 
\[ (v_{j_i} , v_{p+1}) = (t_{j_i} - 2(s - \sum_{j=1}^{i-1} w_j), 2(s - \sum_{j=1}^{i-1} w_j)). \]
The weight of $(b_{j_i},b_{p+1})$ changes from $(t_{j_i}, 0)$ to $(0, t_{j_i})$ as $s$ goes from $\sum_{j=1}^{i-1} w_j$ to $\sum_{j=1}^i w_j$. This means that the map pushes a face of a simplex through the simplex and onto another face. For $\sum_{i=1}^k w_i \leq s \leq 1$, define $f(s, [\sigma, (t_o, \ldots ,t_p)])$ to be constant, equal to 
\[f\left(\sum_{i=1}^k w_i, [\sigma, (t_0, \ldots, t_p)]\right).\] 
In particular $f(1, [\sigma, (t_0, \ldots, t_p)])$ lies in the face of $r_k(\sigma)$ which is in $Star(v)$. The map is continuous since going to a face of $\sigma$ corresponds to a $t_i$ and any corresponding $w_j$ going to zero.
\end{proof}

We now prove \cref{thm: fern connectivity} by using the contractibility of $FA_n^\xi$. For technical reasons we will need to make use of the following. Given a $p$-simplex $\sigma$ of $FA_n^\xi$ (or $A_n^\xi$), denote by $\surf_\sigma$ the subspace\footnote{$\surf_\sigma$ may not be a surface (with boundary) because of the point $*$. On the other hand, this will not be a problem since many of our constructions will still work in this setting.} $\surf - (\sigma - *) \subset \surf$. That is, it is the surface, $\surf$, with a representative fern of $\sigma$ removed, except for the point $*$ on the boundary. For spaces of the form $\surf_\sigma$, we define $A_n^\xi(\surf_{\sigma}, \Delta_1, \ldots, \Delta_n)$ and $FA_n^\xi(\surf_\sigma, \Delta_1, \ldots, \Delta_n)$ as in \cref{def: fern} and \cref{def: fan fern} , with $\surf$ replaced with $\surf_\sigma$ in both definitions, where the marked points of $\surf_\sigma$ are inherited from $\surf$ and the point $* \in \surf_\sigma$ is the point $* \in \del \surf$. The same arguments as in the proof of \cref{lem: fern contractible} can also be used to show that $FA_n^\xi(\surf_\sigma, \Delta_1, \ldots, \Delta_n)$ is contractible, which we will make use of in the following proof of \cref{thm: fern connectivity}.

\begin{proof}[Proof of \cref{thm: fern connectivity}] Recall we are trying to prove that $A_n^\xi(S, \Delta_1, \ldots, \Delta_n)$ is $n-2$ connected. We will actually prove the theorem for $\surf$ a surface or of the form $\surf_\sigma$, where $\sigma \in (A_n^\xi)_p$ is a $p$-simplex. The proof will proceed by induction on $n$. The base case of our induction requires us to show that $A_1^\xi(\surf, \Delta_1)$ is nonempty which is true as long as $\surf$ is connected.

For the inductive step, let $k \leq n -2$. Consider a map 
	\[ f: S^k \rightarrow A_n^\xi. \]
We want to show that $f$ factors through a $(k+1)$-disk. By contractibility of $FA_n^\xi$ we have a commutative diagram
	\[ \xymatrix{
		S^k \ar[r]^{f\;\;} \ar@{^{(}->}[d] &	A_n^\xi \ar@{^{(}->}[d] \\
		D^{k+1} \ar[r]^{\hat{f}\;\;}		&	FA_n^\xi \\
	}	\]
By simplicial approximation, we can triangulate $S^k$ and $D^{k+1}$ and take our maps to be simplicial. We want to deform $\hat{f}$ so that its image lies in $A_n^\xi$. A fern is \emph{coloured by $\Delta_i$} if the endpoints of its arcs are at $\Delta_i$. Call a simplex $\sigma = \langle a_0, \ldots, a_p \rangle \in D^{k+1}$ bad if the colour of each $\hat{f}(a_i)$ already appears as a colour of one of the $\hat{f}(a_j)$ for $i \neq j$.

Let $\sigma \in D^{k+1}$ be a bad simplex of maximal dimension, say $p$. Then $\hat{f}$ restricts to a map
	\[ \hat{f} |_{\mathrm{Link}(\sigma)} : \mathrm{Link}(\sigma) \rightarrow J_\sigma := A_{n'}^\xi(\surf_\sigma, \Delta'_0, \ldots, \Delta'_{n'}) \]
where the $\Delta'_i$ are the remaining partitions of marked points not at the endpoints of ferns  in $\hat{f}(\sigma)$. The image of the map $\hat{f}$ restricts to $A_{n'}^\xi(\surf_\sigma, \Delta'_0, \ldots, \Delta'_{n'})$ because if $\tau \in \mathrm{Link}(\sigma)$ was bad, then $\sigma * \tau$ would be a bad simplex in $D^{k+1}$ of larger dimension than $\sigma$ contradicting maximality of $\sigma$.

Now $n' \geq n - \lf (p+1)/2 \rf$, since a bad $p$-simplex can use up at most $\lf (p+1)/2 \rf$ $\Delta_i$'s. Also note that $p \geq 1$ since there are no bad vertices. By induction the connectivity of $J_\sigma$ is
	\begin{align*}
		\mathrm{conn}(J_\sigma)	&= n' - 2\\
							&\geq n - \lf \frac{p+1}{2} \rf - 2 \\
							&\geq n - 2 - p + \lf \frac{p-1}{2} \rf\\
							& \geq k - p.
	\end{align*}
Since the link of $\sigma$ is a $k-p$ sphere, we have a commutative diagram

\[	\xymatrix{
	\mathrm{Link}(\sigma) \ar[r] \ar@{^{(}->}[d]	&	J_\sigma \ar[r]	&	A^\xi_n(\surf, \Delta_0, \ldots \Delta_n) \\
	K \ar[ur]_{\hat{f}'} & &
	} \]

where $K$ is a $(k-p+1)$-disk with boundary $\del K = \mathrm{Link}(\sigma)$ and the right map is the map that identifies arcs on $\surf'$ with arcs on $\surf$.
In the triangulation of $D^{k+1}$, replace the $(k+1)$-disk, $\mathrm{Star}(\sigma) = \sigma * \mathrm{Link}(\sigma)$ with $\del \sigma * K$. We can do this since $\del \sigma *K$ and $\mathrm{Star}(\sigma)$ have the same boundary. Modify the map $\hat{f}$ by
	\[ \hat{f} * \hat{f}': \del \sigma * K \rightarrow FA_n^\xi(\surf, \Delta_0, \ldots, \Delta_n). \]
New simplices in $\del \sigma * K$ are of the form $\tau = \alpha * \beta$, where $\alpha$ is a proper face of $\sigma$ and $\beta$ is mapped to $J_\sigma$. Thus, if $\tau$ is a bad simplex in $\del \sigma * K$ then $\tau = \alpha$ since ferns of $\hat{f}'(\beta)$ do not share colours with other ferns of $\hat{f}'(\beta)$ or $\hat{f}(\alpha)$, so cannot contribute to a bad simplex. Since $\alpha$ is a proper face of $\sigma$, we have decreased the number of top dimensional bad simplices. The result follows by inducting on top dimensional bad simplices.\end{proof}

\section{The stabilisation map} \label{sec: stab}

In this section, we will define the stabilisation maps of \cref{thm: surface braid stability}. The basic example that one should have in mind is the stabilisation map for braid groups $\stab: \braid_n \rightarrow \braid_{n+1}$ which adds an unbraided strand.

Let $\surf$ be the interior of a surface $\bar{\surf}$ with boundary $\del \bar{\surf}$. Pick a point $b \in \del_0 \subset \del \bar{\surf}$, where $\del_0$ is a boundary component of $\del \bar{\surf}$.
On the level of ordered $\xi$-configuration spaces, we define a map 
	\[ s :\pconf_n^\xi(\surf) \rightarrow \pconf_{n+1}^\xi(\surf'), \]
where $\surf'$ is obtained from $\bar{\surf}$ by adding a collar neighbourhood, $\del_0 \times I,$ around $\del_0$. Let $\bm{q}$ be the subset of the collared neighbourhood given by 
	\[ \bm{q} = \{ (b, \frac{1}{\xi+1}), \ldots, (b, \frac{\xi}{\xi+1}) \}.\]
Then $s$ is defined by inserting this new subset $(\bm{p}_1, \ldots, \bm{p}_m) \mapsto (\bm{p}_1, \ldots, \bm{p}_m, \bm{q}).$ Picking an isomorphism $f: \surf \cong \surf'$, with support in a small neighbourhood of $\del_0$, we get a map $f \circ s :\pconf_n^\xi(\surf) \rightarrow \pconf_{n+1}^\xi(\surf').$
Define the \emph{stabilisation map} for $\xi$-configuration spaces to be the map
	 \[ \stab: \conf_n^\xi(\surf) \rightarrow \conf_{n+1}^\xi(\surf)\]
obtained by quotienting $f \circ s$ by the symmetric group action, that is, it is the map 
	\[ \stab  : \frac{\pconf_n^\xi(\surf)}{\sym_n}  \rightarrow \frac{\pconf_{n+1}^\xi(\surf)}{\sym_{n+1}}.\]


\begin{remark} The stabilisation map for $\xi$-configurations spaces that we have defined corresponds to the map on the level of surface braid groups $\stab: \braid_n^\xi(\surf) \rightarrow \braid_{n+1}^\xi(\surf)$
which adds $\xi$ trivial strands. In particular, $\braid_n^\xi(\surf)$ includes into $\braid_{n+1}^\xi(\surf)$ as the inclusion of a stabiliser of a fern ending at the added marked points of $\conf_{n+1}^\xi(\surf)$.
\end{remark}
\begin{figure}
\begin{center}
\includegraphics[scale=0.40]{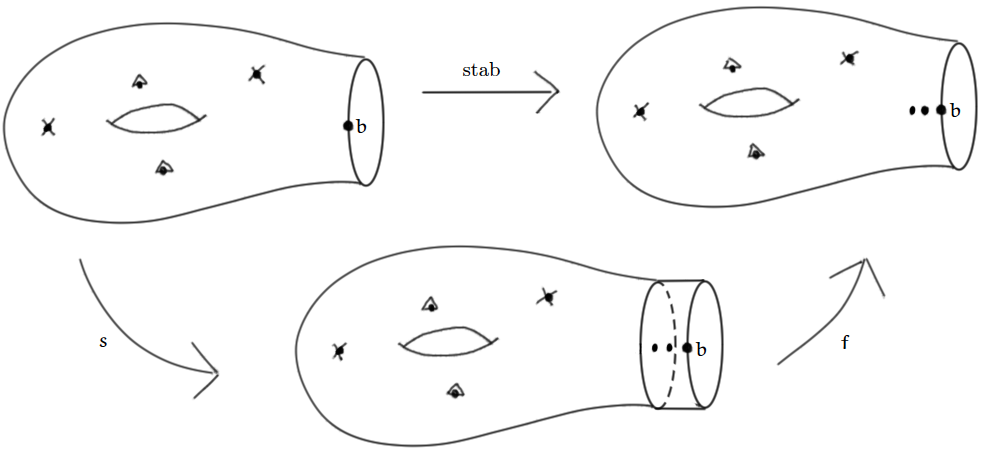}
\caption[The stabilisation map for $\xi$-configurations]{The stabilisation map $\stab: \conf_n^\xi(\surf) \rightarrow \conf_{n+1}^\xi(\surf)$ for $\xi=2$ and $n=2$}\label{fig: stab}
\end{center}
\end{figure}

\section{Homological stability for open surfaces} \label{sec: open}
The goal of this section is to prove \cref{thm: surface braid stability}. For convenience we recall the statement of the theorem.
\begin{theorem}\label{thm: surface braid stability} Let $\surf$ be a surface that is the interior of a surface with boundary. The map 
	\[ \stab_* : H_*(\braid_n^\xi(\surf); \mathbb{Z}) \rightarrow H_*(\braid_{n+1}^\xi(\surf); \mathbb{Z}) \]
is an isomorphism for $2* \leq n$.
\end{theorem}

\begin{proof}
For brevity, we write $\braid_n^\xi$ for $\braid_n^\xi(\surf)$. By Theorem 5.1 of \cite{hw10}, it suffices to check the following conditions to prove that we have isomorphisms in the range $* \leq n/2-1$ and surjections for $* \leq n/2$.
	\begin{enumerate}
		\item The action of $\braid_n^\xi$ on $A_n^\xi$ is transitive on vertices; the stabiliser of each simplex fixes the simplex point wise; and $H_i(A_n^\xi / \braid_n^\xi) = 0$ for $1 \leq i \leq n - 2$.
		\item The subgroup of $\braid_n^\xi$ fixing a $p$-simplex pointwise is conjugate to $\braid_{n - \xi(k + 1)}^\xi$ for some $k \leq p$.
		\item For each edge of $A_n^\xi$ with vertices $v$ and $w$ there exists an element of $\braid_n^\xi$ that takes $v$ to $w$ and that commutes with elements of $\braid_n^\xi$ that leave the edge fixed pointwise.
	\end{enumerate}
	
For the first condition, note that a vertex of $A_n^\xi$ can be identified uniquely with $\xi$-simplices in $A_{n\xi}$, where $A_{n\xi} := A_{n\xi}^1$ is the arc complex of a disk. It is well known that the action of $\braid_n$ on $(A_{n\xi})$ is transitive. Moreover, if the simplices have all arcs going to the same $\Delta_i$, it is possible to choose the element of $\braid_n$ take one simplex to another to preserve the $\Delta_i$ so that it is in $\braid_n^\xi$, which shows that the action of $\braid_n^\xi$ on $A_n^\xi$ is transitive on vertices.

The stabiliser of a simplex fixes the simplex pointwise as the order in which arcs appear at the base point cannot be changed. To see that $H_i(A_n^\xi / \braid^\xi_n)= 0$ for $1 \leq i \leq n-2$ we use an argument from the proof of Lemma 3.3 in \cite{harer85}. The aim is to consider the chain complex $\mathbb{Z}(A_n^\xi / \braid^\xi_n)_p$ (we will describe the differentials momentarily) and construct a chain nullhomotopy.

It is possible to identify the orbit of a $p$-simplex $\sigma$ so that there is a correspondence between $A_n^\xi / \braid^\xi_n$ and orderings of the multiset $\{ 0, \ldots, 0, \ldots, p, \ldots, p\}$ where the first $0$ appears before the first $1$, the first $1$ appears before the first $2$ and so on and each number appears $\xi$ times. For example $(0,1,1,2,0,2)$ is an allowed ordering but $(0,2,0,1,1,2)$ is not. To see this correspondence, we put an ordering on the vertices of $A_n^\xi$ so that for two ferns $\sigma$ and $\tau$ in minimal position, $\sigma < \tau$ if the leftmost arc of $\sigma$ is further to the left than the leftmost arc of $\tau$. Here by leftmost we refer to the left to right ordering of the tangent vectors of the arcs at the base point $*$. With this ordering, $A_n^\xi$ is an ordered simplicial complex. Let $\langle a_0, \ldots a_p \rangle$ be a $p$-simplex in $A_n^\xi$. Labelling all the arcs of $a_i$ by the number $i$ and reading these labels off from left to right at $*$ gives rise to the corresponding ordering of the multiset $\{ 0, \ldots, 0, \ldots, p, \ldots, p\}$. Two ferns that give rise to different orderings of multisets are in different $\braid_n^\xi$ orbits since the action of $\braid_n^\xi$ cannot change the ordering of arcs at $*$. On the other hand, simplices giving rise to the same ordering are in the same orbit by an argument similar to the transitivity argument for vertices. See \cref{pic: fern hom} for a picture of two ferns in the same orbit.
\begin{center}
\begin{figure}
		\includegraphics[scale=0.28]{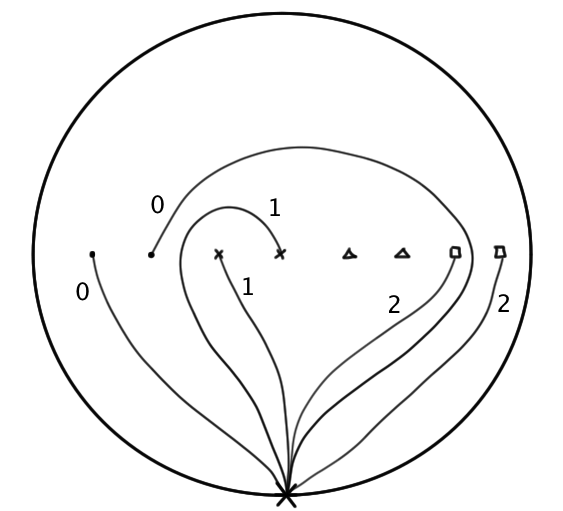}
		\includegraphics[scale=0.28]{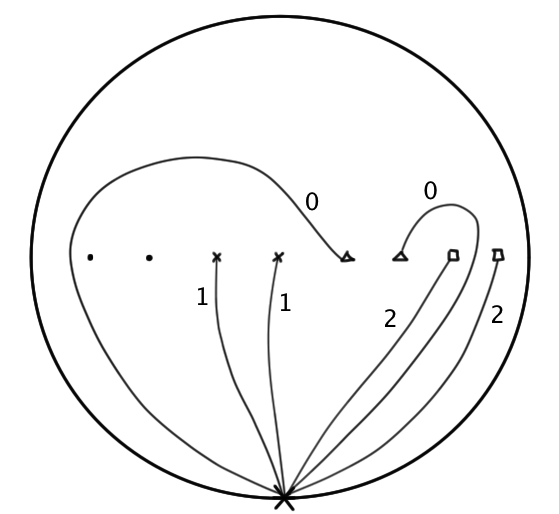}
\caption{Two ferns in the orbit labelled by $(0, 1, 1, 2, 0 , 2)$} \label{pic: fern hom}
\end{figure}
\end{center}
\vspace{-25pt}

For $ 0 \leq i \leq p$, with the above description of simplices in the ordered fern complex $A_n^\xi$, the face maps of $A_n^\xi$ as an ordered simplicial complex are given by
	\[ f_i (\sigma) = \forget_i(\sigma) \]
where $\forget_i$ forgets the $i$'s in $\sigma$ and then subtracts $1$ from all the numbers greater than $i$. For example, $f_1(0, 0, 2, 3, 2, 1,1, 3) = (0,0, 1, 2, 1, 2)$. Thus the differentials of the chains on $A_n^\xi$, $d : \mathbb{Z}(A_n^\xi / \braid_n^\xi)_p \rightarrow \mathbb{Z}(A_n^\xi / \braid_n^\xi)_{p-1}$ are given by the map
	\[ d(\sigma) = \sum_j (-1)^j f_j(\sigma) . \]
For $0 \leq p \leq n - \xi$, define $D : (A_n^\xi / \braid^\xi_n)_p \rightarrow (A_n^\xi / \braid_n^\xi)_{p+1}$ by taking a $p$-simplex $\sigma$, adding one to every entry and then putting $\xi$ zeroes in front. For example $D(0 , 0  ,2 , 2, 1, 1) = (0, 0 , 1, 1, 3, 3, 2, 2)$. We now have that $Dd + d D = id$ so the identity map is chain homotopic to zero in the range $0 \leq p \leq n - \xi$. Thus $H_i(A_n^\xi / \braid^\xi_n) = 0 $ for $1 \leq i \leq n-\xi$.

For the second condition, we get the conjugation by identifying the stabiliser of a $p$-simplex with the $\xi$-braid group that acts on the cut surface where we cut along the arcs of the $p$-simplex.

Finally for the last condition, the action of $\braid^\xi_n$ that takes a vertex $v$ to a vertex $w$ is supported on a tubular neighbourhood of $v \cup w$.

Theorem 5.1 of \cite{hw10} now implies that we have isomorphisms in the range $* \leq n/2 - 1$ and an epimorphism for $* \leq n/2$. In \cref{prop: surface braid transfer}, we will in fact show that the maps 
 \[ \stab: \braid_n^\xi(\surf) \rightarrow \braid_{n+1}^\xi(\surf) \]
are always injective in homology and this gives the full result.\end{proof}

\section{Homological stability for closed surfaces} \label{sec: braid stability closed}
When $\surf$ is a closed surface, it is not possible to define the stabilisation maps that we have been using to prove our homological stability theorems. The main issue is that we no longer have a boundary from which to push in new points. Indeed, even for the usual configuration space, integral homological stability for configuration spaces of closed manifolds does not hold (see for example \cite{fvb62} from which one can compute $H_1(\conf_n(S^2); \mathbb{Z}) \cong \mathbb{Z}/(2n-2)$). On the other hand if one considers rational homology, then Church and Randal-Williams \cite{church12, orw13} prove that homological stability for configuration spaces of closed manifolds holds rationally.

In this section, we will prove the analogous result for closed surfaces and $\xi$-configuration spaces (or equivalently $\xi$-braid groups). We use a technique of Randal-Williams in \cite{orw13} where he proves homological stability for configuration spaces of closed manifolds, using homological stability for open ones. This will involve defining transfer maps which realise the homology isomorphisms. These maps will be more naturally defined in terms of $\xi$-confiugration spaces rather than $\xi$-braids. We define these transfer maps as follows.

For $m < n$, let $\conf_{n,m}^\xi(\surf)$ be the $\xi$-configuration space where the $n$ subsets have been partitioned into subsets of size $m$ and $n-m$. There is a covering map
	\[ p : \conf_{n,m}^\xi(\surf) \rightarrow \conf_n^\xi(\surf) \]
obtained by forgetting the partitioning.
There is a transfer map
	\[ t_{n,m} :H_*(\conf_n^\xi(\surf) ;\mathbb{Z}) \rightarrow H_*(\conf_{n,m}^\xi(\surf); \mathbb{Z}) \rightarrow H_*(\conf_m^\xi(\surf); \mathbb{Z}) \]
where the first map is the transfer map associated to the covering $p$, and the second is the map induced by 
	\[ f: \conf_{n,m}^\xi(\surf) \rightarrow \conf_m^\xi(\surf),\]
which forgets down to $m$ subsets. We use the notation $t_n$ to denote $t_{n, n-1}$

\begin{remark} There is also a way to describe the transfer map on the level of spaces using the identification $H_*(X) \cong \pi_*(\symp^\infty(X))$ from the Dold-Thom theorem. Given an $n$-fold covering $\pi: E \to B$, there is a map $\tau : B \rightarrow \symp^n(E)$,
given by sending a point in $b$ to $\pi^{-1}(b)$. The transfer map on homology is defined as the map $H_*(B) \xrightarrow{\tau_*} H_*(E)$, given by the induced map on $\pi_*$ of $\symp^\infty (\tau) : \symp^\infty B \rightarrow \symp^\infty (E)$,
where we have identified $\symp^\infty(\symp^n(E))$ with $\symp^\infty(E)$ and then used the Dold-Thom theorem. We will make use of this description of the transfer map in the proof of \cref{thm: surface braid closed}.
\end{remark}

\begin{proposition}\label{prop: surface braid transfer} Let $\surf$ be a surface that is the interior of a surface with boundary. The map
	\[ (\stab_n)_* : H_*(\conf_n^\xi(\surf) ; \mathbb{Z}) \rightarrow H_*(\conf_{n+1}^\xi(\surf); \mathbb{Z}) \]
is always split injective. 
Moreover, the transfer map $t_n$ is a rational isomorphism whenever $\stab_{n-1}$ is.
\end{proposition}
\begin{proof}
For brevity let $s_n$ denote the stabilisation map $(\stab_n)_*$.
From the definition of the transfer, we see that the maps $s$ and $t$ satisfy the relations
	\[ t_j \circ s_{j-1} = s_{j-2} \circ t_{j-1} + id. \]
More generally, they satisfy
	\[ t_{j,k}\circ s_{j-1} = s_{j-2} \circ t_{j-1, k-1} + t_{j-1, k}. \]
Furthermore,
	\[ t_{k+1} \circ \cdots \circ t_j = (j-k)!t_{k,j}. \]
Letting $A_j = H_i(\conf_n^\xi(\surf); \mathbb{Z})$ and $B_j := \mathrm{coker}(s_{j-1}),$ we are now in the situation of \cite[Lemma 2]{Dold62}, which implies that the $s_j$ are split injective and that
	\[ t_{j+1} \circ s_j \]
is multiplication by a nonzero constant. On homology, this is rationally an isomorphism so $t_{j+1}$ is an isomorphism whenever $s_j$ is an isomorphism which gives the desired bound.
\end{proof}

We have now completed the proof of \cref{thm: surface braid stability}, using the transfer maps to show that $\stab: \conf_n^\xi(\surf) \rightarrow \conf_{n+1}^\xi(\surf)$ is always injective and is rationally inverse to the transfer map in a range. 

Since the transfer map does not require us to add points to our surface, it makes sense to talk about the transfer map even for closed manifolds. In this way, we get a map $t_n:H_*(\conf_n^\xi(\surf)) \rightarrow H_*(\conf_{n-1}^\xi(\surf))$ and we can ask when this map is an isomorphism. In general, stability does not hold with integral coefficients. However, if we instead work with rational coefficients then we can prove the following.

\begin{theorem} \label{thm: surface braid closed} Let $\surf$ be a surface, open or closed. 
	\[ t : H_k(\conf_n^\xi(\surf) ; \mathbb{Q}) \rightarrow H_k(\conf_{n-1}^\xi(\surf) ; \mathbb{Q}) \]
 is an isomorphism for $2k \leq n-1 $.
\end{theorem}

The proof of \cref{thm: surface braid closed} will be similar to the proof of Proposition 9.4 in \cite{orw13}. We will make use of the following construction.

Let
	\[ D_n^\xi(\surf)_i := \{( \underline{c}, p_0, \ldots, p_i) \in \conf^\xi_n(\surf) \times \surf^{i+1} \st p_j \not\in \underline{c } \mbox{ and }  p_j \neq p_k \}. \]
The notation $p_j \not\in \underline{c}$ means the points $p_j$ do not lie in any of the subsets that make up $\underline{c}$. There are face maps $\del_j : D_n^\xi(\surf)_i \rightarrow D_n^\xi(\surf)_{i-1}$ obtained from forgetting $p_j$. Moreover there is an augmentation map $\epsilon: D_n^\xi(\surf)_\bullet \rightarrow \conf_n^\xi(\surf)$ which forgets all the $p_j$ so that $D_n^\xi(\surf)_\bullet$ is an augmented semi-simplicial space. 

We can similarly define the augmented semi-simplicial space
 \[ D_{n,1}^\xi(\surf)_i := \{( \underline{c}, p_0, \ldots, p_i) \in \conf^\xi_{n,1}(\surf) \times \surf^{i+1} \st p_j \not\in \underline{c } \mbox{ and }  p_j \neq p_k \}. \]

We will make use of the following lemma which follows from the proof of Lemma 2.1 in \cite{orw10}.
\begin{lemma} Let $\norm{X_\bullet} \rightarrow X$ be an augmented (semi-)simplicial space, $f: X_n \rightarrow X$ be the unique face map and let $x \in X$. If $f$ is a Serre fibration, then
	\[ \norm{ f^{-1}(x)_\bullet} \rightarrow \norm{X_\bullet} \rightarrow X \]
is a homotopy fibre sequence. \label{lem: fibre sequence}
\end{lemma}

\begin{lemma} \label{lem: braid weak equiv}The maps $\norm{D_n^\xi(\surf)_\bullet} \to \conf_n^\xi(\surf)$ and $\norm{D_{n,1}^\xi(\surf)_\bullet} \to \conf_{n,1}^\xi(\surf)$ are weak equivalences.
\end{lemma}
\begin{proof} 
We prove that the first map is a weak equivalence.
By \cref{lem: fibre sequence} the map $ \norm{D_n^\xi(\surf)_\bullet} \rightarrow \conf_n^\xi(\surf) $
is a homotopy fibre sequence with fibre over a $\xi$-configuration $\underline{c}$ given by $\norm{F(\surf - \underline{c})_\bullet}$, where $F(\surf - \underline{c})_\bullet$ is the semi-simplicial space whose $i${th} term consists of unordered $(i+1)$-tuples of distinct points in $\surf - \underline{c}$. We will show that $\norm{F(\surf - \underline{c})_\bullet}$ is contractible.

By taking small neighbourhoods of the points in $\underline c$, we can find a closed surface $\surf' \subset (\surf - \underline{c})$ which is homotopy equivalent to $(\surf - \underline{c})$ with some point $x \in (\surf - \underline{c}) - \surf'$. 

Now suppose we have map $f : S^k \to \norm{F(\surf - \underline{c})_\bullet}$. By the previous homotopy equivalence, we can deform $f$ so that $x$ does not lie in its image. Moreover, by simplicial approximation, we can find a PL triangulation of $S^k$ and take $f$ to be simplicial. Now, we can fill in $f$ by defining a map $\hat{f} : \mathrm{cone}(S^k) \to \norm{F(\surf - \underline{c})_\bullet}$ that sends the cone point to $x$.

Therefore $\norm{F( \surf - \underline{c})_\bullet}$ is contractible and so the map $\norm{D_n^\xi(\surf)_\bullet} \rightarrow \conf_n^\xi(\surf)$ is an equivalence. The argument that $\norm{D_{n,1}^\xi(\surf)_\bullet} \to \conf_{n,1}^\xi(\surf)$ is a weak equivalence is similar.
\end{proof}

There are semi-simplicial maps
\[ D_n^\xi(\surf)_\bullet \leftarrow D_{n,1}^\xi(\surf)_\bullet \rightarrow D_{n+1}^\xi(\surf)_\bullet \]	
modelled on the maps for $\xi$-configurations, $\conf_n^{\xi}(\surf) \leftarrow \conf_{n,1}^\xi(\surf) \rightarrow \conf_{n+1}^\xi(\surf)$.

\begin{lemma} \label{lem: braid resolution} There is a map of semi-simplicial spaces
\[ D_n^\xi(\surf)_\bullet \xrightarrow{\tau} \symp^n_{fib}(D_{n,1}^\xi(\surf)_\bullet) \rightarrow \symp^n_{fib} (D^\xi_{n+1}(\surf)_\bullet) \]
which induces a map
\[ \norm{D_n^\xi(M)_\bullet} \xrightarrow{\tau} \symp^n(\norm{D_{n,1}^\xi(\surf)_\bullet}) \rightarrow \symp^n(\norm{D_{n+1}^\xi(\surf)_\bullet}) \]
on geometric realisations.
The terms in these maps are described in the proof.
\end{lemma}
\begin{proof} There are fibration sequences given by
	\[ \conf_n^\xi(\surf - \mbox{ $i+1$ points }) \rightarrow D_n^\xi(\surf)_i \xrightarrow{\pi} \pconf_{i+1}(\surf) \]
where $\pi$ is the map that sends $(\underline{c}, p_0, \ldots, p_i) \mapsto (p_0, \ldots, p_i)$. Let $\symp^n_{fib}(D_n^\xi(\surf)_i)$ denote the $n$-fold fibrewise symmetric product with respect to this fibration. The maps $D_{n,1}^\xi(\surf)_i \rightarrow D_{n+1}^\xi(\surf)$ are fibrewise $n$-fold coverings over $\pconf_{i+1}(\surf)$ which give rise to transfer maps 
	\[ \tau : D_{n+1}^\xi(\surf)_i \rightarrow \symp^n_{fib} (D_{n,1}^\xi(\surf)_i).\] 
One can check that $\tau$ commutes with the face maps and so defines a transfer $\tau$ on geometric realisations as in the statement of the lemma.
\end{proof}

We can now prove \cref{thm: surface braid closed}
\begin{proof}[Proof of \cref{thm: surface braid closed}] Associated to a semi-simplicial space is a spectral sequence which computes the homology of its geometric realisation in terms the the homology of its levels. Therefore \cref{lem: braid resolution} gives rise to a map of spectral sequences converging to the transfer map
	\[ t: H_*( \conf_n^\xi(\surf) ; \mathbb{Q}) \rightarrow H_*(\conf_{n-1}^\xi(\surf); \mathbb{Q}) \]
where we have used \cref{lem: braid weak equiv} to identify $\conf_n^\xi(\surf) \simeq \norm{D_n^\xi(\surf)}$ and the Dold-Thom theorem to get a map from $H_*(\norm{D_n^\xi(\surf)_\bullet} ; \mathbb{Q}) \rightarrow H_*(\norm{D_{n-1}^\xi(\surf)_\bullet}; \mathbb{Q})$.

The map on $E^1_{pq}$ terms of the spectral sequence is
	\[ H_q(D_n^\xi(\surf)_p ; \mathbb{Q}) \rightarrow H_q(D_{n-1}^\xi(\surf)_p ; \mathbb{Q})\]
and is induced by $D_n^\xi(M)_i \rightarrow \symp^n_{fib}(D_{n,1}(\surf)_i) \rightarrow \symp^n_{fib}(D_{n-1}^\xi(\surf)_p)$. From the fibration in the proof of \cref{lem: braid resolution}, there is a Serre spectral sequence converging to this map. The map on $E^2_{pq}$ terms of this Serre spectral sequence is
\begin{align*} H_p(\pconf_i(\surf); H_q( \conf_n^\xi(&\surf -  \mbox{$i+1$ points}); \mathbb{Q})) \rightarrow \\
		& H_p(\pconf_{i}(\surf); H_q( \conf_{n-1}^\xi(\surf - \mbox{$i+1$ points}); \mathbb{Q})). 
\end{align*}
On coefficients, it is induced by the transfer map
	\[t:  H_q( \conf_n^\xi(\surf -  \mbox{$i+1$ points}); \mathbb{Q}) \rightarrow H_q( \conf_{n-1}^\xi(\surf - \mbox{$i+1$ points}); \mathbb{Q}). \]
Since $(\surf - \mbox{$i + 1$ points})$ is an open surface, \cref{prop: surface braid transfer} implies that this is an isomorphism for $q \leq (n-1)/2$. Thus the map $H_q( D_n^\xi(\surf)^p; \mathbb{Q}) \rightarrow H_q( D_{n-1}^\xi(\surf)^p; \mathbb{Q})$ is an isomorphism in this range. In particular,
	\[ t_n: H_*(\conf_n^\xi(\surf) ;\mathbb{Q}) \rightarrow H_*(\conf_{n-1}^\xi(\surf) ; \mathbb{Q})  \]
is an isomorphism for $* \leq (n-1)/2$.
\end{proof}

\bibliographystyle{alpha}
\bibliography{references}

\end{document}